\documentclass[12pt]{amsart}
\usepackage[mathscr]{eucal}
\usepackage[dvips]{graphicx}

\theoremstyle{plain}
	\newtheorem{theorem}{Theorem}[section]
	\newtheorem{lemma}[theorem]{Lemma}

	\newtheorem{proposition}[theorem]{Proposition}

\theoremstyle{plain}
	\newtheorem{maintheorem}{Theorem}

\def\R{\mathbb{R}}

\def\bbS{\mathbb{S}}

\def\calF{\mathcal{F}}
\def\calG{\mathcal{G}}
\def\calH{\mathcal{H}}

\def\e{\varepsilon}
\def\bx{\bar{x}}
\def\by{\bar{y}}
\def\tx{\tilde{x}}
\def\ty{\tilde{y}}
\def\tW{\tilde{W}}
\def\bW{\bar{W}}

\makeatletter
 \@addtoreset{equation}{section}
\makeatother

\begin{document}

% **********************************************************
% Title of This Paper
% **********************************************************

\title[Nonradial singular solutions]{Infinitely many nonradial singular solutions\\ of $\Delta u+e^u=0$ in $\R^N\backslash\{0\}$, $4\le N\le 10$}
\thanks{This work was supported by JSPS KAKENHI Grant Numbers 24740100, 16K05225.}

\author{Yasuhito Miyamoto}
\address{Graduate School of Mathematical Sciences, The University of Tokyo,
3-8-1 Komaba, Meguro-ku, Tokyo 153-8914, Japan}
\email{miyamoto@ms.u-tokyo.ac.jp}

\begin{abstract}
We construct countably infinitely many nonradial singular solutions of the problem
\[
\Delta u+e^u=0\ \ \textrm{in}\ \ \R^N\backslash\{0\},\ \ 4\le N\le 10
\]
of the form
\[
u(r,\sigma)=-2\log r+\log 2(N-2)+v(\sigma),
\]
where $v(\sigma)$ depends only on $\sigma\in \bbS^{N-1}$.
To this end we construct countably infinitely many solutions of
\[
\Delta_{\bbS^{N-1}}v+2(N-2)(e^v-1)=0,\ \ 4\le N\le 10,
\]
using ODE techniques.
\end{abstract}
\date{\today}
\maketitle

%\begin{keyword}
%Global branch \sep Bifurcation \sep Nodal curve \sep Nodal domain\\
%\PACS 
%{\it 2000 MSC:} 35B32 \sep 35J25 \sep 35P30
%\end{keyword}

%\end{frontmatter}

% **********************************************************
% **********************************************************
% **********************************************************
% Section 1
% **********************************************************
% **********************************************************
% **********************************************************
\section{Introduction and Main results}
We are concerned with the problem
\begin{equation}\label{E}
\Delta U+e^U=0\ \ \textrm{in}\ \ \R^N\backslash\{0\}.
\end{equation}
If $N\ge 3$, then this problem has the radial solution
\begin{equation}\label{SSR}
U^*(R)=-2\log R+\log 2(N-2),
\end{equation}
which is singular at the origin.
The aim of the paper is to construct infinitely many nonradial classical solutions of (\ref{E}) such that solutions are singular only at the origin.
Let $\R_+:=\{ x\in\R;\ x>0\}$, and let $\bbS^{N-1}$ denote the $(N-1)$-dimensional unit sphere.
Let $(R,\sigma)\in\R_+\times\bbS^{N-1}$ be the spherical polar coordinates.
We will find solutions of the form
\begin{equation}\label{SS}
U(R,\sigma)=-2\log R+\log 2(N-2)+V(\sigma)
\end{equation}
such that $V$ is smooth.
The main result of the paper is the following:
\begin{maintheorem}\label{A}
Suppose that $4\le N\le 10$.
The problem (\ref{E}) has countably infinitely many nonradial classical solutions of the form (\ref{SS}) such that these solutions diverge at the origin.
Here, $V(\sigma)$ is nonconstant.
\end{maintheorem}
Substituting (\ref{SS}) into (\ref{E}), we see that $V$ satisfies
\begin{equation}\label{V}
\Delta_{\bbS^{N-1}}V+2(N-2)(e^V-1)=0\ \ \textrm{on}\ \ \bbS^{N-1},
\end{equation}
where $\Delta_{\bbS^{N-1}}$ denotes the Laplace-Beltrami operator on $\bbS^{N-1}$.
Theorem~\ref{A} immediately follows from Theorem~\ref{B} below.
\begin{maintheorem}\label{B}
Suppose that $4\le N\le 10$.
The problem (\ref{V}) has countably infinitely many axially symmetric nonconstant classical solutions.
\end{maintheorem}
In this paper we mainly study (\ref{V}).

When $N=3$, Bidaut-V\'eron {\it et al.} \cite{MBL91} studied nonradial singular solutions of (\ref{E}) and other equations.
The equation (\ref{V}) becomes $\Delta_{\bbS^2}V+2(e^V-1)=0$.
This is called the conformal Gaussian curvature equation, and this equation and related equations have been studied for more than three decades.
All regular solutions of (\ref{V}) are described in \cite{CY87,O82}.
In particular, axially symmetric solutions can be written explicitly as $V(\theta)=-2\log(\sqrt{c^2+1}-c\cos\theta)$, where $c\in\R$ is constant and $\theta\in [0,\pi]$ is the geodesic distance from the north pole of $\bbS^2$.
Hence,
\begin{equation}\label{SS3}
U(R,\theta)=-2\log R+\log 2-\log (\sqrt{c^2+1}-c\cos\theta)
\end{equation}
is a one parameter family of nonradial singular solutions of (\ref{E}) in the case $N=3$.
The singular solution (\ref{SS3}) can be seen as a singular solution of the Dirichlet problem
\[
\begin{cases}
\Delta U+e^U=0 & \textrm{in}\ \ \Omega\backslash\{0\}\\
U=0 & \textrm{on}\ \ \partial\Omega,
\end{cases}
\]
where $\Omega:=\{ U>0\}\subset\R^3$.

Nontrivial one-point singular solutions of the equation $\Delta U+e^U=0$ were constructed by several authors when the domain is bounded.
In \cite{R99b} R\'eba\"{i} studied nonradial singular solutions in the case $N=3$.
Let $B_{r}$ denote the ball centered at the origin with radius $r>0$.
She showed, among other things, that there is small $\e>0$ such that if $\xi_0\in B_{\e}$, then the problem
\[
\begin{cases}
\Delta U+\lambda e^U=0 & \textrm{in}\ \ B_1\backslash\{\xi_0\}\\
U=0 & \textrm{on}\ \ \partial B_1
\end{cases}
\]
has a singular solution for some $\lambda>0$ provided that $N=3$.
In particular, this singular solution is not radially symmetric.
Note that the same result was announced by H.~Matano and his method is different from \cite{R99b}.
In \cite{DD07} D\'avila and Dupaigne constructed a singular solution when the domain is close to the unit ball provided that $N\ge 4$.
Specifically, they showed that if $N\ge 4$ and $t>0$ is small, then the problem
\[
\begin{cases}
\Delta U+\lambda e^U=0 & \textrm{in}\ \ D_t\backslash\{\xi(t)\}\\
U=0 & \textrm{on}\ \ \partial D_t
\end{cases}
\]
has a singular solution $(\lambda(t),U(x,t))$ such that as $t\rightarrow 0$,
\[
\left\|U(x,t)-(-2\log{|x-\xi(t)|})\right\|_{L^{\infty}(D_t)}\rightarrow 0\qquad\textrm{and}\qquad\lambda(t)\rightarrow 2(N-2),
\]
where $D_t:=\{ x+t\psi(x);\ x\in B_1\subset\R^N\}$ and $\psi$ is a $C^2$-mapping from $\bar{B}_1$ to $\R^N$. 
%***************************************************************************
Solutions with finitely or infinitely many singularities were constructed by Pacard~\cite{P90} and Horshin~\cite{H95,H97} when $N>10$, and by R\'eba\"i~\cite{R99a} when $N=3$.
Our solutions given by Theorem~\ref{A} are candidates of the asymptotic profiles of those singular solutions near a singular point.

Similar problems were studied for the equation $\Delta U+\lambda (1+U)^p=0$ in \cite{DD07,R99b} and the equation $\Delta U+U^p=0$ in \cite{DGW12}.
In particular, Dancer {\it et al.} \cite{DGW12} constructed infinitely many nonradial positive singular solutions of the Lane-Emden equation
\begin{equation}\label{LE}
\Delta U+U^p=0\ \ \textrm{in}\ \ \R^N\backslash\{0\}
\end{equation}
with
\begin{equation}\label{DGWC}
\frac{N+1}{N-3}<p<p_{JL}(N-1)\quad\textrm{and}\quad N\ge 4,
\end{equation}
where $p_{JL}(M)$ is defined by
\[
p_{JL}(M):=
\begin{cases}
1+\frac{4}{M-4-2\sqrt{M-1}} & \textrm{if}\ M\ge 11,\\
\infty & \textrm{if}\ 2\le M\le 10.
\end{cases}
\]
%They constructed singular solutions of (\ref{LE}) of the form
Let us consider the solution of the form
\[
U(R,\sigma)=R^{-\frac{2}{p-1}}V(\sigma).
\]
Then $V$ satisfies the equation
\begin{equation}\label{YP}
\Delta_{\bbS^{N-1}}V-\mu V+V^p=0,
\end{equation}
where
\[
\mu:=\frac{2}{p-1}\left( N-2-\frac{2}{p-1}\right).
\]
If $p<(N+1)/(N-3)$, then all solutions are constant \cite{MBL91}.
When $p=(N+1)/(N-3)$, $\mu=(N-3)(N-1)/4$.
Then, (\ref{YP}) becomes Yamabe problem on $\bbS^{N-1}$ and various solutions are known.
If (\ref{DGWC}) holds, then in \cite{DGW12} Dancer {\it et al.} showed that (\ref{YP}) has infinitely many nonconstant regular positive solutions.
Theorem~\ref{B} in the present paper is its exponential counterpart.

Let us mention technical details.
We construct nonconstant regular solutions of (\ref{V}).
We use an ODE approach.
Specifically, we find solutions $v$ in the space of functions depending only on $\theta\in[0,\pi]$ which is the geodesic distance from the north pole of $\bbS^{N-1}$.
Since the solution is regular at both the north and south poles, it holds that $v'(0)=v'(\theta)=0$.
If it is not the case, then $v$ is singular at $\theta=0$ or $\pi$, and hence the solution (\ref{SS}) becomes singular along a (half) line.
$v$ satisfies
%The problem becomes the following:
\begin{equation}\label{S1E1}
\begin{cases}
v''+(N-2)\frac{\cos\theta}{\sin\theta}v'+2(N-2)(e^v-1)=0, & 0<\theta<\pi,\\
v'(0)=v'(\pi)=0.
\end{cases}
\end{equation}
%since the solution of (\ref{S1E1}) becomes the solution of (\ref{YP}).
If $v(\theta)$ satisfies (\ref{S1E1}), then $v(\pi-\theta)$ also satisfies (\ref{S1E1}).
We will use a shooting method from $\theta=0$.
In this case it is difficult to analyze the behavior of the solution near $\theta=\pi$, because of the singularity of the equation at $\theta=\pi$.
In order to avoid the difficulty we find symmetric solutions, i.e., $v(\theta)=v(\pi-\theta)$.
Then (\ref{S1E1}) becomes the following:
\begin{equation}\label{S1E2}
\begin{cases}
v''+(N-2)\frac{\cos\theta}{\sin\theta}v'+2(N-2)(e^v-1)=0, & 0<\theta<\frac{\pi}{2},\\
v'(0)=v'\left(\frac{\pi}{2}\right)=0.
\end{cases}
\end{equation}
If $N\ge 4$, then (\ref{S1E2}) has the exact singular solution
\begin{equation}\label{S1E3}
v^*(\theta):=-2\log\sin\theta+\kappa_{N-1},\ \ \textrm{where}\ \ \kappa_{N-1}:=\log\frac{N-3}{N-2}.
\end{equation}
In the Euclidean case it is well known that the equation $U''+\frac{N-1}{R}U'+e^U=0$ can be changed into a homogeneous equation in $R$ by the transformation
\[
X(T)=U(R)-U^*(R)\quad\textrm{and}\quad T:=\log R.
\]
Therefore, a phase plane analysis is applicable.
See (\ref{S2E4}) in Section~2.
%works well, where $U^*(R)$ is the singular solution defined by (\ref{SSR}).
Essentially the same transformation works well for the problem (\ref{S1E2}).
Hereafter we consider the case $N\ge 4$. Then $v^*(\theta)$ is well defined.
Let 
\begin{equation}\label{S1E3+}
x(t):=v(\theta)-v^*(\theta)\quad\textrm{and}\quad t:=\log\tan\frac{\theta}{2}.
\end{equation}
This change of variables is natural, since $dt/d\theta=1/\sin\theta$.
$x(t)$ satisfies
\begin{equation}\label{S1E4}
\begin{cases}
x''-(N-3)\tanh(t)x'+2(N-3)(e^x-1)=0, & -\infty<t<0,\\
\cosh(t)x'(t)+2\sinh(t)\rightarrow 0\ \ \textrm{as}\ \ t\rightarrow -\infty,\\
x'(0)=0,
\end{cases}
\end{equation}
where we use the equality $v'(\theta)=\cosh(t)x'(t)+2\sinh(t)$.
The method so far is the same as the case $\Delta U+U^p=0$ used in \cite{DGW12}.
However, our method of the construction of solutions of (\ref{S1E4}) is different from that of \cite{DGW12} which uses the matched asymptotic expansions.
Our proof is elementary and shorter.
Using our method, one can obtain the main result of \cite{DGW12}, i.e., 
the existence of infinitely many positive solutions of (\ref{YP}).
%(\ref{YP}) has infinitely many positive radial solutions if (\ref{DGWC}) holds.
See Section~4 of the present paper.
We use a phase plane analysis, although (\ref{S1E4}) is not homogeneous.
%The effect of this inhomogeneity can be reduced by a scaling argument.
%A regular perturbation method and the winding number of the orbit $(x(t),x'(t))$ play important roles.
When $t$ is negatively large, the equation in (\ref{S1E4}) is close to the homogeneous equation (\ref{S2E1}) below.
Perturbing a solution of the homogeneous equation, we find a solution of (\ref{S1E4}) that satisfies the boundary conditions of (\ref{S1E2}).
This method is inspired by that of \cite{MP91}.
However, the authors of \cite{MP91} used a technical argument of the uniform convergence to the solution of the limit equation.
% instead of a regular perturbation method.

This paper consists of four sections.
In Section~2 we recall known results of radial solutions of $\Delta U+e^U=0$.
In Section~3 we prove Theorem~\ref{B} which leads to Theorem~\ref{A}.
In Section~4 we briefly prove the existence of infinitely many positive radial solutions of (\ref{YP}), using our method.

% **********************************************************
% **********************************************************
% **********************************************************
% Section 2
% **********************************************************
% **********************************************************
% **********************************************************
\section{Preliminaries}
In this section we recall known results about the following equation of $\bx(s)$:
\begin{equation}\label{S2E1}
\bx''+(N-3)\bx'+2(N-3)(e^{\bx}-1)=0,\quad -\infty<s<\infty.
\end{equation}
This is the limit equation of (\ref{S1E4}) as $t\rightarrow -\infty$.
% and it also appears in the problem
This equation is equivalent to the radial equation of the original equation (\ref{E}) in one dimension less, i.e.,
%\begin{equation}\label{S2E2}
\[
\begin{cases}
\Delta U+e^U=0 & \textrm{in}\ \ \R^{N-1},\\
U\ \textrm{is radial.}
\end{cases}
\]
%\end{equation}
%First, we derive (\ref{S2E1}) from (\ref{S2E2}).
%We consider the initial value problem
To see this we consider the problem
\begin{equation}\label{S2E3}
\begin{cases}
U''+\frac{N-2}{R}U'+e^U=0, & 0<R<\infty,\\
U(0)=\bar{\alpha},\\
U'(0)=0.
\end{cases}
\end{equation}
The equation in (\ref{S2E3}) has the singular solution $U^*(R)=-2\log R+\bar{\kappa}_{N-1}$, where $\bar{\kappa}_{N-1}:=\log 2(N-3)$.
We define $\bx(s)$ and $s$ by
\[
\bx(s):=U(R)-U^*(R)\qquad\textrm{and}\quad s:=\log R,
\]
respectively.
The problem (\ref{S2E3}) becomes
\begin{equation}\label{S2E4}
\begin{cases}
\bx''+(N-3)\bx'+2(N-3)(e^{\bx}-1)=0, & -\infty<s<\infty,\\
\bx(s)-2s+\bar{\kappa}_{N-1}-\bar{\alpha}\rightarrow 0\ \ \textrm{as}\ \ s\rightarrow -\infty,\\
e^{-s}(\bx'(s)-2)\rightarrow 0\ \ \textrm{as}\ \ s\rightarrow -\infty,
\end{cases}
\end{equation}
where we use the equalities $U(R)=\bx(s)-2s+\bar{\kappa}_{N-1}$ and $U'(R)=e^{-s}(\bx'(s)-2)$.

We use a phase plane argument.
Let $\by(s):=\bx'(s)$.
By (\ref{S2E4}) we obtain the following:
\begin{equation}\label{S2E5}
\begin{cases}
\bx'=\by, & -\infty<s<\infty,\\
\by'=-(N-3)\by-2(N-3)(e^{\bx}-1), & -\infty<s<\infty,\\
\bx(s)-2s+\bar{\kappa}_{N-1}-\bar{\alpha}\rightarrow 0\ \ \textrm{as}\ \ s\rightarrow -\infty,\\
e^{-s}(\by(s)-2)\rightarrow 0\ \ \textrm{as}\ \ s\rightarrow -\infty.
\end{cases}
\end{equation}
Various properties of the solution $(\bx(s),\by(s))$ of (\ref{S2E5}), which we call the orbit, are known.
We summarize these properties of this orbit in the following proposition:
\begin{proposition}\label{S2P1}
Assume that $N\ge 4$.
The (\ref{S2E5}) has the unique entire solution.
The orbit $\{(\bx(s),\by(s));\ -\infty<s<\infty\}$ in the $xy$-plane starts along the line $y=2$ at $s=-\infty$ and converges to the origin as $s\rightarrow \infty$.
When $4\le N\le 10$, the origin is a stable spiral and the orbit rotates clockwise around the origin.
Therefore, there is $\{s_j\}_{j=1}^{\infty}$ $(s_1<s_2<\cdots\rightarrow\infty)$ such that $\by(s_j)=0$ $(s\in\{1,2,\ldots\})$ and
\begin{equation}\label{S2P1E0}
\bx(s_2)<\bx(s_4)<\cdots<\bx(s_{2j})<\cdots<0<\cdots<\bx(s_{2j-1})<\cdots<\bx(s_3)<\bx(s_1).
\end{equation}
See Figure~\ref{fig1}.
\end{proposition}
\begin{figure}
\begin{center}
\includegraphics[scale=1.0]{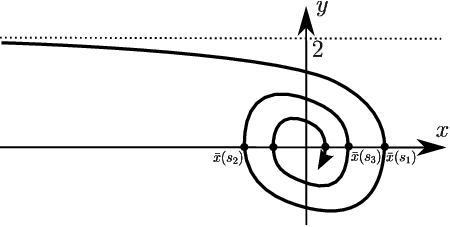}
\caption{A schematic picture of the phase plane in the case $4\le N\le 10$.}
\label{fig1}
\end{center}
\end{figure}
We briefly prove Proposition~\ref{S2P1} for readers' convenience.
\begin{proof}
We omit the proof of the existence and uniqueness of the solution.
We prove other properties of the orbit which are later used in the proof of the main theorem.
Because of (\ref{S2E5}), $\lim_{s\rightarrow-\infty}(\bx(s)-2s)=-\bar{\kappa}_{N-1}+\bar{\alpha}$ and $\lim_{s\rightarrow-\infty}(\by(x)-2)=0$.
Thus the orbit starts along the line $y=2$ at $s\rightarrow-\infty$.
The problem (\ref{S2E5}) has the Lyapunov function
\begin{equation}\label{S2P1E1--}
E(x,y):=\frac{y^2}{2}+2(N-3)(e^x-x).
\end{equation}
Indeed, we have
\begin{equation}\label{S2P1E1-}
\frac{d}{ds}E(\bx(s),\by(s))=-(N-3)\by^2(s)\le 0.
\end{equation}
We show by contradiction that the problem
\begin{equation}\label{S2P1E1}
\begin{cases}
\bx'=\by\\
\by'=-(N-3)\by-2(N-3)(e^{\bx}-1)
\end{cases}
\end{equation}
does not have a nontrivial periodic orbit.
Assume that (\ref{S2P1E1}) has a nontrivial periodic orbit.
Then we see by (\ref{S2P1E1-}) that $\by(s)\equiv 0$.
Because of (\ref{S2P1E1}), $\bx'(s)\equiv 0$ and $\bx(s)$ is constant.
The orbit $(\bx(s),\by(s))$ is an equilibrium of (\ref{S2P1E1}).
This contradicts the assumption.
Hence, (\ref{S2P1E1}) does not have a nontrivial periodic orbit.

Let
\begin{equation}\label{S2P1E1+}
\Omega_c:=\{(x,y);\ E(x,y)<c\}.
\end{equation}
For each large $c>0$, $\Omega_c$ is a bounded set in the $xy$-plane.
For large $c>0$, there is $s_0\in\R$ such that $\{(\bx(s),\by(s))\}_{s\ge s_0}\subset\Omega_c$, and $\{(\bx(s),\by(s))\}_{s\ge s_0}$ is bounded.
Because (\ref{S2P1E1}) does not have a periodic orbit, by the Poincar\'e-Bendixson theorem we see that the orbit converges to the origin which is the unique equilibrium.

In order to study the behavior of the orbit near the origin we consider the linearized problem of (\ref{S2P1E1}) at the origin, i.e.,
\[
\left(
\begin{array}{cc}
0 & 1\\
-2(N-3) & -(N-3)
\end{array}
\right).
\]
The two eigenvalues $\lambda_{\pm}$ of the matrix are given by the characteristic equation $\lambda^2+(N-3)\lambda+2(N-3)=0$.
We have
\[
\lambda_{\pm}:=\frac{1}{2}\left\{ -(N-3)\pm\sqrt{(N-3)(N-11)}\right\}.
\]
If $4\le N\le 10$, then the two eigenvalues are complex with negative real part.
Hence, the origin is a stable spiral.
We see by the direction of the vector field defined by (\ref{S2P1E1}) that the orbit rotates clockwise around the origin and (\ref{S2P1E0}) holds.
The proof is complete.
\end{proof}

% **********************************************************
% **********************************************************
% **********************************************************
% Section 3
% **********************************************************
% **********************************************************
% **********************************************************
\section{Proof of Theorem~\ref{B}}
In order to find solutions of (\ref{S1E2}) we study the problem
\begin{equation}\label{S3E1}
\begin{cases}
v''+(N-2)\frac{\cos\theta}{\sin\theta}v'+2(N-2)(e^v-1)=0, & 0<\theta<\frac{\pi}{2},\\
v(0)=\alpha,\\
v'(0)=0,
\end{cases}
\end{equation}
where $\alpha\in\R$ is a parameter.
We find a classical solution $v(\theta)\in C^2[0,\frac{\pi}{2}]$.
%We call $v(\theta)$ the solution of (\ref{S3E1}) if
%\begin{equation}\label{S3E1+-}
%v(\theta)\in C^2(0,\frac{\pi}{2}]\cap C^1[0,\frac{\pi}{2}]
%\end{equation}
%and if $v(\theta)$ satisfies (\ref{S3E1}).
%We also study the problem
%\begin{equation}\label{S3E1+}
%\begin{cases}
%v''+(N-2)\frac{\cos\theta}{\sin\theta}v'+2(N-2)(e^v-1)=0, & 0<\theta<\frac{\pi}{2},\\
%\lim_{\theta\downarrow 0}v(\theta)=\alpha,\\
%\lim_{\theta\downarrow 0}v'(\theta)=0.
%\end{cases}
%\end{equation}
%We call $v(\theta)$ a solution of (\ref{S3E1+}) if
%\begin{equation}\label{S3E1++}
%v(\theta)\in C^2(0,\frac{\pi}{2}]
%\end{equation}
%and if $v(\theta)$ satisfies (\ref{S3E1+}).
%If $v(\theta)$ is the solution of (\ref{S3E1}), then the restricted function of $v(\theta)$ is the solution of (\ref{S3E1+}).
%We assume that $v(\theta)$ is the solution of (\ref{S3E1+}).
%We define $v(0)=\alpha$ so that $v(\theta)$ is defined on $[0,\frac{\pi}{2}]$ and continuous at $\theta=0$.
%Using L'Hospital's rule, we have
%\[
%v'(0)=\lim_{\theta\downarrow 0}\frac{v(\theta)-\alpha}{\theta}=\lim_{\theta\downarrow 0}\frac{v'(\theta)}{1}=0.
%\]
%Since $v'(0)=0=\lim_{\theta\downarrow 0}v'(\theta)$, $v'(\theta)\in %C^0[0,\frac{\pi}{2}]$.
%We see that $v(\theta)\in C^1[0,\frac{\pi}{2}]$ and $v(\theta)$ is the solution of (\ref{S3E1}).
%Thus, the solution of (\ref{S3E1+}) can be uniquely extended as the solution of (\ref{S3E1}).
%The problem (\ref{S3E1}) is equivalent to (\ref{S3E1+}).
%Hence, we consider (\ref{S3E1+}).

\begin{lemma}\label{S3L0}
Let $x(t)$ be defined by (\ref{S1E3+}).
The function $v(\theta)$ is the solution of (\ref{S3E1}) if and only if
\begin{equation}\label{S3L0E1+}
x(t)\in C^2(-\infty,0]
\end{equation}
and $x(t)$ satisfies
\begin{equation}\label{S3L0E2}
\begin{cases}
x''-(N-3)\tanh(t)x'+2(N-3)(e^x-1)=0, & -\infty<t<0,\\
x(t)+2\log\cosh(t)+\kappa_{N-1}-\alpha\rightarrow 0\ \ \textrm{as}\ \ t\rightarrow -\infty,\\
\cosh(t)x'(t)+2\sinh(t)\rightarrow 0\ \ \textrm{as}\ \ t\rightarrow -\infty.
\end{cases}
\end{equation}
\end{lemma}
We omit the proof.
%\begin{proof}
%By direct calculation we see that the equation in (\ref{S3E1+}) is equivalent to that of (\ref{S3L0E2}).
%Let $v(\theta)$ be the solution of (\ref{S3E1+}).
%It follows from the definition of $x(t)$ that (\ref{S3L0E1+}) holds if and only if (\ref{S3E1++}).
%Using (\ref{S1E3+}) and the equality $v'(\theta)=\cosh(t)x'(t)+2\sinh(t)$, we see that the initial conditions of (\ref{S3E1+}) are equivalent to those of (\ref{S3L0E2}).
%The proof is complete.
%\end{proof}
We call $x(t)$ the solution of (\ref{S3L0E2}) if (\ref{S3L0E1+}) and (\ref{S3L0E2}) hold.

If the solution of (\ref{S3L0E2}) satisfies $x'(0)=0$, then this solution satisfies (\ref{S1E4}).
Therefore, the function $v(\theta)$, which is associated to $x(t)$ by (\ref{S1E3+}), becomes a solution of (\ref{S1E2}).
\begin{lemma}\label{S3L1}
Let
\begin{equation}\label{S3L1E1}
\tx(s):=x(t)\quad\textrm{and}\quad s:=t+\frac{\alpha}{2}.
\end{equation}
%Then the problem (\ref{S3E2}) is equivalent to the following problem:
The function $x(t)$ is a solution of (\ref{S3L0E2}) if and only if
%$\tx(s)$ is a solution of the problem
\begin{equation}\label{S3L1E2-}
\tx(s)\in C^2(-\infty,\frac{\alpha}{2}]
\end{equation}
and $\tx(s)$ satisfies
\begin{equation}\label{S3L1E2}
\begin{cases}
\tx''-(N-3)\tanh\left(s-\frac{\alpha}{2}\right)\tx'+2(N-3)(e^{\tx}-1)=0, & -\infty<s<\frac{\alpha}{2},\\
\tx(s)-2s-2\log2+\kappa_{N-1}\rightarrow 0\ \ \textrm{as}\ \ s\rightarrow-\infty,\\
e^{-s}(\tx'(s)-2)\rightarrow 0\ \ \textrm{as}\ \ s\rightarrow-\infty.
\end{cases}
\end{equation}
\end{lemma}

\begin{proof}
It is clear that the equation in (\ref{S3L0E2}) is equivalent to that in (\ref{S3L1E2}).
We check the equivalence of initial conditions.
Let $x(t)$ be a solution of (\ref{S3L0E2}).
Since
\[
\lim_{s\rightarrow-\infty}\left(2\log\cosh\left(s-\frac{\alpha}{2}\right)-2\log\frac{e^{-s}}{2}-\alpha\right)=\lim_{s\rightarrow-\infty}2\log (e^{2s-\alpha}+1)=0,
\]
\begin{multline*}
0=\lim_{t\rightarrow-\infty}(x(t)+2\log\cosh(t)+\kappa_{N-1}-\alpha)\\
=\lim_{s\rightarrow-\infty}\left\{\tx(s)+2\log\frac{e^{-s}}{2}+\kappa_{N-1}+\left(2\log\cosh\left(s-\frac{\alpha}{2}\right)-2\log\frac{e^{-s}}{2}-\alpha\right)\right\}\\
=\lim_{s\rightarrow-\infty}\left(\tx(s)-2s-2\log 2+\kappa_{N-1}\right).
\end{multline*}
Since
\[
0=\lim_{t\rightarrow-\infty}(\cosh(t)x'(t)+2\sinh(t))=\lim_{s\rightarrow-\infty}\left\{\frac{1}{2}(e^{-s+\frac{\alpha}{2}}+e^{s-\frac{\alpha}{2}})(\tx'(s)-2)+2e^{s-\frac{\alpha}{2}}\right\},
\]
\[
0=\lim_{s\rightarrow-\infty}\left|\frac{1}{2}(e^{-s+\frac{\alpha}{2}}+e^{s-\frac{\alpha}{2}})(\tx'(s)-2)\right|\ge\frac{e^{\frac{\alpha}{2}}}{2}\lim_{s\rightarrow-\infty}|e^{-s}(\tx'(s)-2)|.
\]
Hence, $\lim_{s\rightarrow-\infty}e^{-s}(\tx'(s)-2)=0$.
Thus $\tx(s)$ satisfies (\ref{S3L1E2}).
We can check that the converse is also true.
We omit the detail.
The proof is complete.
\end{proof}
We call $\tx(s)$ the solution of (\ref{S3L1E2}) if (\ref{S3L1E2-}) and (\ref{S3L1E2}) hold.

Let $\tx(s)$ be the solution of (\ref{S3L1E2}), and let $\ty(s):=\tx'(s)$.
The pair of functions $(\tx(s),\ty(s))$ satisfies
\begin{equation}\label{S3E3}
\begin{cases}
\tx'=\ty, & -\infty<s<s_0,\\
\ty'=(N-3)\tanh\left(s-\frac{\alpha}{2}\right)\ty-2(N-3)(e^{\tx}-1), & -\infty<s<s_0,\\
\tx(s)-2s+\tilde{\kappa}_{N-1}\rightarrow 0\ \ \textrm{as}\ \ s\rightarrow-\infty,\\
e^{-s}(\ty(s)-2)\rightarrow 0\ \ \textrm{as}\ \ s\rightarrow-\infty,
\end{cases}
\end{equation}
where $\tilde{\kappa}_{N-1}:=\kappa_{N-1}-2\log 2$ and $s_0=\alpha/2$.
We study the behavior of the orbit $(\tx(s),\ty(s))$ when $\alpha$ is large.
Since $\alpha$ is large, $\tanh\left(s-\frac{\alpha}{2}\right)$ is close to $-1$.
We can expect that $(\tx(s),\ty(s))$ behaves like the solution of (\ref{S2E5}) with $\bar{\alpha}=\bar{\kappa}_{N-1}-\kappa_{N-1}+2\log 2$.

\begin{lemma}\label{S3L2}
Let $s_0\in\R$ be fixed.
Let
\[
u(r):=\tx(s)-2s+\tilde{\kappa}_{N-1}\quad\textrm{and}\quad r:=e^s.
\]
The pair of functions
\begin{equation}\label{S3L2E0}
(\tx(s),\ty(s))\in C^2(-\infty,s_0]\times C^1(-\infty,s_0]
\end{equation}
satisfies (\ref{S3E3}) if
\begin{equation}\label{S3L2E1-}
u(r)\in C^2(0,r_0]\cap C^1[0,r_0]
\end{equation}
and $u(r)$ satisfies the problem
\begin{equation}\label{S3L2E1}
\begin{cases}
u''+\frac{N-2}{r}u'+8(N-2)e^u-\frac{2(N-3)\delta}{1+\delta r^2}(ru'+2)=0, & 0<r\le r_0,\\
u(0)=0,\\
u'(0)=0,
\end{cases}
\end{equation}
where $r_0:=e^{s_0}$ and $\delta=e^{-\alpha}$.
\end{lemma}
We omit the proof.
%\begin{proof}
%The proof is almost the same as that of Lemma~\ref{S3L0}.
%\end{proof}
Lemma~\ref{S3L2} shows that the limit equation of (\ref{S3L2E1}) as $\delta\downarrow 0$ is $u''+\frac{N-2}{r}u'+8(N-2)e^u=0$.
%The change of variables (\ref{S3L1E1}) corresponds to a self
%Since $u(r)+\alpha=v(\theta)+4\log\cos\frac{\theta}{2}$ and $r=e^{\frac{\alpha}{2}}\tan\frac{\theta}{2}$, Lemma~\ref{S3L2} corresponds to the scaling argument.

We call $u(r)$ the solution of (\ref{S3L2E1}) if (\ref{S3L2E1-}) and (\ref{S3L2E1}) hold.

\begin{lemma}\label{S3L3}
Let $r_0>0$ be fixed.
Let $u(r,\delta)$ be the solution of (\ref{S3L2E1}).
For each small $\e>0$, there exists $\delta_0>0$ such that if $|\delta|<\delta_0$, then $\left\|u(\,\cdot\,,0)-u(\,\cdot\,,\delta)\right\|_{C^1[0,r_0]}<\e$.
\end{lemma}

\begin{proof}
We define $\calF(u,\delta)$ and $\calG(u,\delta)$ by
\begin{align*}
\calF(u(r),\delta)&:=\int_0^r\left(\frac{1}{s^{N-2}}\int_0^s\calG(u(\tau),\delta)d\tau+\frac{2(N-3)\delta s}{1+\delta s^2}u(s)\right)ds\quad\textrm{and}\\
\calG(u(r),\delta)&:=-8(N-2)r^{N-2}e^{u(r)}-2(N-3)\delta\frac{N-1+(N-3)\delta r^2}{(1+\delta r^2)^2}r^{N-2}u(r)\\
&\quad+\frac{4(N-3)\delta}{1+\delta r^2}r^{N-2},
\end{align*}
respectively. First, we show that if $u\in C^0[0,r_0]$ satisfies
\begin{equation}\label{S3L3E1}
u(r)=\calF(u(r),\delta)\ \ \textrm{for}\ \ r\in [0,r_0],
\end{equation}
then $u$ is a solution of (\ref{S3L2E1}).
Let $u\in C^0[0,r_0]$ be a function such that (\ref{S3L3E1}) holds.
We immediately see that $u(0)=0$.
Since $\calF(u(r),\delta)\in C^1(0,r_0]$, $u\in C^1(0,r_0]$.
Differentiating (\ref{S3L3E1}) with respect to $r$, we have
\begin{equation}\label{S3L3E2}
u'(r)=\frac{1}{r^{N-2}}\int_0^r\calG (u(\tau),\delta)d\tau+\frac{2(N-3)\delta r}{1+\delta r^2}u(r).
\end{equation}
Since the right-hand side of (\ref{S3L3E2}) is in $C^1(0,r_0]$, $u\in C^2(0,r_0]$.
Multiplying (\ref{S3L3E2}) by $r^{N-2}$ and differentiating it with respect to $r$, we have
\begin{align*}
(r^{N-2}u')'&=\calG (u(r),\delta)+\left(\frac{2(N-3)\delta r^{N-1}}{1+\delta r^2}u(r)\right) ' \\
&=-8(N-2)r^{N-2}e^u+\frac{2(N-3)\delta}{1+\delta r^2}r^{N-2}(ru'(r)+2).
\end{align*}
Thus, $u$ satisfies the equation in (\ref{S3L2E1}).
There is $C>0$ such that $|\calG(u(\tau),\delta)|\le C\tau^{N-2}$ and $\left|\frac{2(N-3)\delta r}{1+\delta r^2}u(r)\right|\le Cr$,
we have
%\begin{align*}
\[
\left|\frac{1}{r^{N-2}}\int_0^r\calG(u(\tau),\delta)d\tau+\frac{2(N-3)\delta r}{1+\delta r^2}u(r)\right|\le C\left(\frac{1}{N-1}+1\right)r\rightarrow 0\ \ (r\downarrow 0).
\]
%\end{align*}
Therefore, by (\ref{S3L3E2}) we see that $\lim_{r\downarrow 0}u'(r)=0$.
On the other hand, by L'Hospital's rule we have
\[
u'(0)=\lim_{r\downarrow 0}\frac{u(r)-0}{r}=\lim_{r\downarrow 0}\frac{u'(r)}{1}=0.
\]
Since $u'(0)=0=\lim_{r\downarrow 0}u'(r)$, $u\in C^1[0,r_0]$.
Since $u\in C^2(0,r_0]\cap C^1[0,r_0]$, $u$ is the solution of (\ref{S3L2E1}).

Let $\calH(u,\delta):=u-\calF(u,\delta)$.
We study the solution of the problem
\begin{equation}\label{S3L3E2+}
\calH(u,0)=0\ \ \textrm{in}\ \ C^0[0,r_0].
\end{equation}
Then (\ref{S3L3E2+}) is equivalent to (\ref{S3L2E1}) with $\delta=0$.
Let $U(R):=u(r)+\log 8(N-2)$ and $R:=r$.
The problem (\ref{S3L2E1}) with $\delta=0$ is equivalent to the problem (\ref{S2E3}) with $\bar{\alpha}=\log 8(N-2)$.
It is well known that (\ref{S2E3}) has the unique solution $U(R)$ which defined in $R\ge 0$.
Hence, (\ref{S3L3E2+}) has the unique solution $u_0(r)\in C^0[0,r_0]$.
We have
\begin{equation}\label{S3L3E4}
\calH(u_0,0)=0\ \ \textrm{in}\ \ C^0[0,r_0].
\end{equation}

We consider the linearized problem
\[
D_u\calH(u_0(r),0)[\phi(r)]=\phi(r)-D_u\calF(u_0(r),0)[\phi(r)]=0,
\]
where $\phi\in C^0[0,r_0]$.
Since
\[
D_u\calF(u_0,0)[\phi]=\int_0^r\frac{1}{s^{N-1}}\int_0^s(-8(N-2)\tau^{N-2}e^{u_0(\tau)}\phi(\tau) )d\tau ds,
\]
we can by a similar argument show that $\phi\in C^2(0,r_0]\cap C^1[0,r_0]$ and $\phi$ is the solution of the problem
\begin{equation}\label{S3L3E5}
\begin{cases}
\phi''+\frac{N-2}{r}\phi'+8(N-2)e^{u_0}\phi=0, & 0<r\le r_0,\\
\phi(0)=0,\\
\phi'(0)=0.
\end{cases}
\end{equation}
Because of the uniqueness of the solution of (\ref{S3L3E5}), $\phi(r)\equiv 0$ ($0\le r\le r_0$).
Since $D_u\calF:C^0[0,r_0]\rightarrow C^0[0,r_0]$ is compact, by the Fredholm alternative we see that
\begin{equation}\label{S3L3E6}
\textrm{the mapping}\ D_u\calH(u_0,0)[\phi]=\phi-D_u\calF(u_0,0)[\phi]\ \textrm{is invertible.}
\end{equation}

We find a solution near $u_0$.
It is clear that
\begin{equation}\label{S3L3E7}
\calH(u,\delta)\ \textrm{is $C^1$ near $(u_0,0)$.}
\end{equation}
By (\ref{S3L3E4}), (\ref{S3L3E6}), and (\ref{S3L3E7}) we apply the implicit function theorem to $\calH(u,\delta)=0$ at $(u_0,0)$.
There is a small $\delta_0>0$ and a smooth mapping $u=u_{\delta}$ such that if $|\delta|<\delta_0$, then $\calH(u_{\delta},\delta)=0$ in $C^0[0,r_0]$ and $\|u_0-u_{\delta}\|_{C^0[0,r_0]}\rightarrow 0$ as $\delta\rightarrow 0$.
%If $|\delta|$ is small, then $\|u_0-u_{\delta}\|_{C^0[0,r_0]}$ is small, hence there is a small $\eta_0>0$ such that $|\calG(u_0(r),0)-\calG(u_{\delta}(r),\delta)|\le\eta_0r^{N-2}$.
For each small $\eta_0>0$, there is $\delta>0$ such that $|\calG(u_0(r),0)-\calG(u_{\delta}(r),\delta)|\le\eta_0r^{N-2}$ for $r\in[0,r_0]$.
Using this inequality and (\ref{S3L3E2}), we have
\begin{align}
|u'_0(r)-u_{\delta}'(r)|&\le\frac{1}{r^{N-2}}\int_0^{r_0}\eta_0r^{N-2}dr+\delta\left\|\frac{2(N-3)r}{1+\delta r^2}u_{\delta}(r)\right\|_{C^0[0,r_0]}\nonumber\\
&=\frac{\eta_0r_0}{N-1}+\delta\left\|\frac{2(N-3)r}{1+\delta r^2}u_{\delta}(r)\right\|_{C^0[0,r_0]}.\label{S3L3E8}
\end{align}
The inequality (\ref{S3L3E8}) means that $\left\|u'_0-u'_{\delta}\right\|_{C^0[0,r_0]}\rightarrow 0$ as $\delta\downarrow 0$.
Thus, $\left\|u_0-u_{\delta}\right\|_{C^1[0,r_0]}\rightarrow 0$ as $\delta\downarrow 0$ and the proof is complete.
\end{proof}

\begin{lemma}\label{S3L4}
Let $(\bx(s),\by(s))$ be the solution of (\ref{S2E5}) with $\bar{\alpha}=\log 8(N-2)$, and let $(\tx(s),\ty(s))$ be the solution of (\ref{S3E3}).
For each $s_0>0$ and $\e>0$, there is $\alpha_0>0$ such that if $\alpha>\alpha_0$, then
\begin{equation}\label{S3L4E1}
\left\|\bx(\,\cdot\,)-\tx(\,\cdot\,)\right\|_{C^0(-\infty,s_0]}<\e\qquad\textrm{and}\qquad\left\|\by(\,\cdot\,)-\ty(\,\cdot\,)\right\|_{C^0(-\infty,s_0]}<\e.
\end{equation}
\end{lemma}

\begin{proof}
Let $U(R):=\bx(s)-2s+\bar{\kappa}_{N-1}$ and $R:=e^s$.
Then $U(R)$ satisfies (\ref{S2E3}) with $\bar{\alpha}=\log 8(N-2)$.
Let $u_0(r):=U(R)-\log 8(N-2)$ and $r:=R$.
Then $u_0(r)$ satisfies (\ref{S3L2E1}) with $\delta=0$.

Let $u_{\delta}(r):=\tx(s)-2s+\tilde{\kappa}_{N-1}$ and $r:=e^s$.
Then $u_{\delta}(r)$ satisfies (\ref{S3L2E1}).
Because of Lemma~\ref{S3L3}, for small $\e>0$, there is $\delta_0>0$ such that
\begin{equation}\label{S3L4E2}
\textrm{if $|\delta|<\delta_0$, then}\ 
\left\|u_0(\,\cdot\,)-u_{\delta}(\,\cdot\,)\right\|_{C^0[0,r_0]}<\e\ \textrm{and}\ \left\|u_0'(\,\cdot\,)-u'_{\delta}(\,\cdot\,)\right\|_{C^0[0,r_0]}<e^{-s_0}\e,
\end{equation}
where $s_0:=\log r_0$.
Since $u_0(r)=\bx(s)-2s+\bar{\kappa}_{N-1}-\log 8(N-2)$ and $u_{\delta}(r)=\tx(s)-2s+\tilde{\kappa}_{N-1}$, we have
\[
\left\|u_0(\,\cdot\,)-u_{\delta}(\,\cdot\,)\right\|_{C^0(0,r_0]}=\left\|\bx(\,\cdot\,)-\tx(\,\cdot\,)\right\|_{C^0(-\infty,s_0]},
\]
where we use the equality $\bar{\kappa}_{N-1}-\log 8(N-2)-\tilde{\kappa}_{N-1}=0$.
Since $u'_0(r)=e^{-s}(\bx'(s)-2)$ and $u_{\delta}'(r)=e^{-s}(\tx'(s)-2)$,
\begin{equation}\label{S3L4E3}
\left\|u'_0(\,\cdot\,)-u'_{\delta}(\,\cdot\,)\right\|_{C^0(0,r_0]}=\left\|e^{-s}(\bx'(\,\cdot\,)-\tx'(\,\cdot\,))\right\|_{C^0(-\infty,s_0]}\ge e^{-s_0}\left\|\bx'(\,\cdot\,)-\tx'(\,\cdot\,)\right\|_{C^0(-\infty,s_0]}.
\end{equation}
Let $\alpha_0:=-\log\delta_0$.
Combining (\ref{S3L4E3}) and (\ref{S3L4E2}), we see that (\ref{S3L4E1}) holds for $\alpha>\alpha_0$.
%Since $\delta=e^{-\alpha}$, by (\ref{S3L4E2}) we see that if $\alpha>-\log\delta_0$, then (\ref{S3L4E1}) holds.
\end{proof}

%Hereafter, we define $\bar{\alpha}:=\log 8(N-2)$.
Hereafter, by $(\tx(s,\alpha),\ty(s,\alpha))$ we denote the solution of (\ref{S3E3}).
By $(\bx(s),\by(s))$ we denote the solution of (\ref{S2E5}) with $\bar{\alpha}:=\log 8(N-2)$.
Lemma~\ref{S3L4} says that for each fixed $s_0\in\R$, $(\tx(s,\alpha),\ty(s,\alpha))$ is uniformly close to $(\bx(s),\by(s))$ for $s\in(-\infty,s_0]$ if $\alpha$ is large.
Since $(\bx(s),\by(s))$ rotates around the origin infinitely many times (Proposition~\ref{S2P1}), we expect that there are infinitely many $\alpha\in\R$ such that $\ty(\frac{\alpha}{2},\alpha)=0$.
Since $y(0,\alpha)=\ty(\frac{\alpha}{2},\alpha)=0$, $u(r)$ corresponding to $(\tx(s,\alpha),\ty(s,\alpha))$ is the solution of (\ref{S1E2}).
%the orbit $(\tx(s,\alpha),\ty(s,\alpha))$ corresponds to the solution of that of (\ref{S1E2}).
To prove the existence of such $\alpha$ we use the \lq\lq half winding numbers" of the two orbits $\{(\tx(s,\alpha),\ty(s,\alpha))\}$ and $\{(\bx(s),\by(s))\}$.
We define $\tilde{W}_I(\alpha)$ and $\bar{W}_I$ by
\[
\tilde{W}_I(\alpha):=\sharp\{ s\in I;\ \ty(s,\alpha)=0\},
\]
\[
\bar{W}_I:=\sharp\{ s\in I;\ \by(s)=0\},
\]
where $I\subset\R$ is an interval.
For example, it is clear that $\bar{W}_{(-\infty,s_1]}<\infty$ if $|s_1|<\infty$. Proposition~\ref{S2P1} says that $\bar{W}_{(-\infty,\infty)}=\infty$.

\begin{lemma}\label{S3L5}
Let $(\tx(s,\alpha),\ty(s,\alpha))$ be the solution of (\ref{S3E3}).
There is a sequence $\{\alpha_j\}_{j=1}^{\infty}$ $(\alpha_1<\alpha_2<\cdots\rightarrow +\infty)$ such that
\begin{equation}\label{S3L5E0-}
\tW_{(-\infty,\frac{\alpha_1}{2}]}(\alpha_1)<\tW_{(-\infty,\frac{\alpha_2}{2}]}(\alpha_2)<\cdots\rightarrow\infty
\end{equation}
and
\[
\ty(\frac{\alpha_j}{2},\alpha_j)=0\ \ \textrm{for}\ \ j\in\{1,2,\ldots\}.
\]
\end{lemma}

\begin{proof}
First, we show that if $\alpha$ is large, then 
\begin{equation}\label{S3L5E0}
\tW_{(-\infty,\frac{\alpha}{2}]}(\alpha)<\infty.
\end{equation}
Since $(0,0)$ is the equilibrium of the vector field defined by the first order system in (\ref{S3E3}), we see that $(\tx(s,\alpha),\ty(s,\alpha))\neq (0,0)$ for $s\in\R$.
If $(\tx,\ty)$ is on the $x$-axis, then by (\ref{S3E3}) we have
\[
\begin{cases}
\tx'=0\\
\ty'=-2(N-3)(e^{\tx}-1).
\end{cases}
\]
Therefore, $\ty'\neq 0$, since $\tx\neq 0$.
This means that the orbit does not stay on the $x$-axis when it crosses the $x$-axis and that $\tW_{(-\infty,s]}(\alpha)$ increases by one whenever the orbit crosses the $x$-axis.
Let $s_0\in\R$ be fixed.
We take a large $\alpha>0$ such that $s_0<\frac{\alpha}{2}$.
The orbit $(\bx(s),\by(s))$ ($-\infty<s<s_0$) rotates around the origin finitely many times.
Because of Lemma~\ref{S3L4}, $(\tx(s,\alpha),\ty(s,\alpha))$ is uniformly close to $(\bx(s),\by(s))$ for $s\in (-\infty,s_0)$, and hence $(\tx(s,\alpha),\ty(s,\alpha))$ $(-\infty<s<s_0)$ rotates around the origin finitely many times.
We see that $\tW_{(-\infty,s_0)}(\alpha)<\infty$.
We study the behavior of the orbit $(\tx(s,\alpha),\ty(s,\alpha))$ in the interval $[s_0,\frac{\alpha}{2}]$.
Let $E(x,y)$ be as defined by (\ref{S2P1E1--}). Then, for $s\le\frac{\alpha}{2}$,
\begin{align*}
\frac{d}{ds}E(\tx(s,\alpha),\ty(s,\alpha))&=\ty(\ty'+2(N-3)(e^{\tx}-1))\\
&=(N-3)\tanh\left(s-\frac{\alpha}{2}\right)\ty^2\\
&\le 0.
\end{align*}
Thus, $(\tx(s,\alpha),\ty(s,\alpha))$ $(s_0\le s\le\frac{\alpha}{2})$ is in the bounded set $\Omega_{c_0}\subset\R^2$, where $\Omega_c$ is defined by (\ref{S2P1E1+}) and $c_0:=E(\tx(s_0,\alpha),\ty(s_0,\alpha))$.
In particular, $\tx(s,\alpha)$ $(s_0\le s\le\frac{\alpha}{2})$ is bounded.
On the other hand, $\tx(s,\alpha)$ satisfies the following linear ODE of the second order:
\[
\tx''-(N-3)\tanh\left(s-\frac{\alpha}{2}\right)\tx'+2(N-3)V(\tx)\tx=0,
\]
where
\[
V(\tx):=
\begin{cases}
\frac{e^{\tx}-1}{\tx} & \textrm{if}\ \tx\neq 0,\\
1 & \textrm{if}\ \tx=0.
\end{cases}
\]
%Since the interval $[s_0,\frac{\alpha}{2}]$ is compact, $\tx(s,\alpha)$ has at most finitely many critical points in $[s_0,\frac{\alpha}{2}]$.
We show by contradiction that $\tx(s,\alpha)$ has at most finitely many critical points in $[s_0,\frac{\alpha}{2}]$.
Suppose the contrary.
Since $[s_0,\frac{\alpha}{2}]$ is compact, the set of the critical points has an accumulation point $s_1$.
Because of the continuity of $\tx'(s)$, $\tx'(s_1)=0$.
If $\tx(s_1)\neq 0$, then $\tx''(s_1)=-2(N-3)V(\tx(s_1))\tx(s_1)\neq 0$, and hence $s_1$ is a simple zero of $\tx'(s)$.
Therefore, $\tx'(s)\neq 0$ near $s=s_1$ except $s_1$, which contradicts that $\tx'(s)$ has infinitely many zeros near $s_1$.
Thus, $\tx(s_1)=0$.
Because $\tx(s_1)=\tx'(s_1)=0$, it follows from the uniqueness of the solution of the second order ODE that $\tx(s)\equiv 0$.
We obtain the contradiction.
Since $\tx(s,\alpha)$ has at most finitely many critical points in $[s_0,\frac{\alpha}{2}]$, $\ty(s,\alpha)$ also has at most finitely many zeros in $[s_0,\frac{\alpha}{2}]$.
Thus, $\tW_{[s_0,\frac{\alpha}{2}]}<\infty$.
Since $\tW_{(-\infty,\frac{\alpha}{2}]}(\alpha)=\tW_{(-\infty,s_0)}(\alpha)+\tW_{[s_0,\frac{\alpha}{2}]}(\alpha)$, $\tW_{(-\infty,\frac{\alpha}{2}]}(\alpha)<\infty$.

Second, we show that 
\begin{equation}\label{S3L5E2}
\tW_{(-\infty,\frac{\alpha}{2}]}(\alpha)\rightarrow\infty\ \ \textrm{as}\ \ \alpha\rightarrow\infty.
\end{equation}
%We take large $\alpha>0$.
%By the same argument as above we see the following:
Since $\bar{W}_{(-\infty,\infty)}=\infty$,
%For each large $M>0$, there are a large $s_0\in\R$ and a large $\alpha_0(>2s_0)$ such that if $\alpha>\alpha_0$, then $\tW_{(-\infty,s_0)}(\alpha)>M$.
for each large $M>0$, there is a large $s_0\in\R$ such that $\bW_{(-\infty,s_0)}\ge M+1$.
If $\alpha>0$ is large, then it follows from Lemma~\ref{S3L4} that $(\tx(s),\ty(s))$ is uniformly close to $(\bx(s),\by(s))$ in $(-\infty,s_0)$.
Since each zero of $\by(s)$ is simple, $\ty(s)$ also has a zero near a zero of $\by(s)$.
Thus, $\tW_{(-\infty,s_0)}\ge M$ for large $\alpha$, because $\bW_{(-\infty,s_0)}\ge M+1$.
Since $\tW_{(-\infty,\frac{\alpha}{2}]}(\alpha)\ge\tW_{(-\infty,s_0)}(\alpha)$ and $M$ can be chosen arbitrarily large, (\ref{S3L5E2}) holds.

It is clear that $(\tx(\frac{\alpha}{2},\alpha),\ty(\frac{\alpha}{2},\alpha))$ is continuous in $\alpha$.
Because of (\ref{S3L5E2}), (\ref{S3L5E0}), and this continuity, there is a sequence $\{\alpha_j\}_{j=1}^{\infty}$ $(\alpha_1<\alpha_2<\cdots\rightarrow\infty)$ such that $\ty(\frac{\alpha_j}{2},\alpha_j)=0$.
We can choose a subsequence, which is still denoted by $\{\alpha_j\}_{j=1}^{\infty}$, such that (\ref{S3L5E0-}) holds, because of (\ref{S3L5E2}) and (\ref{S3L5E0}).
The proof is complete.
\end{proof}

We are in a position to prove Theorem~\ref{B}.
\begin{proof}[Proof of Theorem~\ref{B}]
Let $\{\alpha_j\}_{j=1}^{\infty}$ be a sequence given in Lemma~\ref{S3L5}, and let $(\tx(s,\alpha),\ty(s,\alpha))$ be the solution of (\ref{S3E3}).
We let
\[
x(t,\alpha_j):=\tx(s,\alpha_j),\ \ y(t,\alpha_j):=\ty(s,\alpha_j),\quad\textrm{and}\quad t:=s-\frac{\alpha_j}{2}.
\]
Then $(x(t,\alpha_j),y(t,\alpha_j))$ is a solution of (\ref{S3L0E2}) and $x'(0,\alpha_j)=y(0,\alpha_j)=0$.
Let $v(\theta,\alpha_j)$ be defined by (\ref{S1E3+}) with $x(t)=x(t,\alpha_j)$.
Since $v'(\theta,\alpha_j)=\cosh(t)x'(t,\alpha_j)-2\sinh(t)/\cosh(t)$, $v'(\frac{\pi}{2},\alpha_j)=0$, and hence $v(\theta,\alpha_j)$ satisfies (\ref{S1E2}).
We define $\tilde{v}(\theta,\alpha_j)$ by
\[
\tilde{v}(\theta,\alpha_j):=
\begin{cases}
v(\theta,\alpha_j) & 0\le\theta\le\frac{\pi}{2},\\
v(\pi-\theta,\alpha_j) & \frac{\pi}{2}<\theta\le\pi.
\end{cases}
\]
Then $\tilde{v}(\theta,\alpha_j)$ $(j\in\{1,2,\ldots\})$ is a classical solution of (\ref{S1E1}).
The proof is complete.
\end{proof}

% **********************************************************
% **********************************************************
% **********************************************************
% Section 4
% **********************************************************
% **********************************************************
% **********************************************************
\section{Infinitely many radial solutions of (\ref{YP})}
In this section we briefly prove the following:
\begin{proposition}[{\cite[Theorem 1.1]{DGW12}}]
If (\ref{DGWC}) holds, then (\ref{YP}) has infinitely many positive radial solutions. Therefore, (\ref{LE}) has infinitely many singular positive nonradial solutions.
\end{proposition}

\begin{proof}
In order to find radial solutions of (\ref{YP}) we study
\begin{equation}\label{S4E1}
\begin{cases}
v''+(N-2)\frac{\cos\theta}{\sin\theta}v'-\mu v+v^p=0, & 0<\theta<\frac{\pi}{2},\\
v(0)=\alpha,\\
v'(0)=0.
\end{cases}
\end{equation}
We find $\alpha>0$ such that $v'(\frac{\pi}{2})=0$.
We define
\[
q:=\frac{2}{p-1},\ \ A:=\{q(N-3-q)\}^{\frac{1}{p-1}},\quad\textrm{and}\quad m:=A^{-\frac{p-1}{2}}.
\]
By direct calculation we see that $v^*(\theta):=A\sin^{-q}\theta$ is a singular solution of the equation in (\ref{S4E1}).
Using the transformation
\[
x(t):=\frac{v(\theta)}{v^*(\theta)}\quad\textrm{and}\quad t:=\frac{1}{m}\log\tan\frac{\theta}{2},
\]
we have
%\begin{equation}\label{S4E1+}
\[
\begin{cases}
x''-(N-3-2q)m\tanh(mt)x'-x+x^p=0, & -\infty<t<0,\\
x(t)A\cosh^q(mt)\rightarrow\alpha\ \ \textrm{as}\ \ t\rightarrow -\infty,\\
\cosh(mt)\frac{d}{dt}\left\{ x(t)A\cosh^q(mt)\right\}\rightarrow 0\ \ \textrm{as}\ \ t\rightarrow -\infty.
\end{cases}
\]
%\end{equtaion}
We use the change of variables
\[
\tilde{x}(s):=x(t)\quad\textrm{and}\quad s:=t+\frac{1}{mq}\log\alpha
\]
which corresponding to (\ref{S3L1E1}).
We have
\begin{equation}\label{S4E2}
\begin{cases}
\tx'=\ty,\\
\ty'=(N-3-2q)m\tanh(ms-\frac{\log\alpha}{q})\tx'+\tx-\tx^p,\\
\tx(s)A\cosh^q(ms)\rightarrow 1\ \ \textrm{as}\ \ s\rightarrow -\infty,\\
\cosh(ms)\frac{d}{ds}\left\{ \tx(s)A\cosh^q(ms)\right\}\rightarrow 0\ \ \textrm{as}\ \ s\rightarrow -\infty.
\end{cases}
\end{equation}
We let $\delta:=\alpha^{-\frac{4}{q}}$.
We consider the case where $\alpha$ is large.
Then $\delta>0$ is small.
Let
\[
\frac{u(r)}{u^*(r)}=\tx(s),\ \ r=e^{ms},\quad\textrm{and}\quad u^*(r):=Ar^{-q}.
\]
Then by (\ref{S4E2}) we have
\[
\begin{cases}
u''+\frac{N-2}{r}u'+u^p-\frac{2(N-3-2q)\delta}{1+\delta r^2}(ru'+qu)=0,\\
u(0)=1,\\
u'(0)=0.
\end{cases}
\]
Note that the limit equation as $\delta\rightarrow 0$ is $u''+\frac{N-2}{r}u'+u^p=0$.
Using the same argument as in Lemmas~\ref{S3L3} and \ref{S3L4}, we can show that for each $s_0\in\R$ if $\alpha$ is large, then the orbit $(\tx(s),\ty(s))$ of (\ref{S4E2}) is close to the orbit $(\bx(s),\by(s))$ of
\begin{equation}\label{S4E3}
\begin{cases}
\bx'=\by,\\
\by'=-(N-3-2q)m\by+\bx-\bx^p,\\
\bx(s)A\cosh^q(ms)\rightarrow 1\ \ \textrm{as}\ \ s\rightarrow -\infty,\\
\cosh(ms)\frac{d}{ds}\left\{ \bx(s)A\cosh^q(ms)\right\}\rightarrow 0\ \ \textrm{as}\ \ s\rightarrow -\infty.
\end{cases}
\end{equation}
in the interval $(-\infty,s_0]$.
It is well known that the orbit $(\bx(s),\by(s))$ is a heteroclinic orbit between the two equilibria $(0,0)$ and $(1,0)$ under the condition where $N-3-2q>0$.
The linearization of (\ref{S4E3}) at $(1,0)$ is
\[
\left(
\begin{array}{cc}
0 & 1\\
1-p & -(N-3-2q)m
\end{array}
\right).
\]
The two eigenvalues of the matrix are complex if $(N-3-2q)^2m^2-4(p-1)<0$.
In this case $\frac{1}{2}(N-5-\sqrt{N-2})<q<\frac{N-3}{2}$ which is equivalent to (\ref{DGWC}).
Thus if (\ref{DGWC}) holds, then $(1,0)$ becomes a spiral.
Using the same argument as in Lemma~\ref{S3L5}, we can show that there is a sequence $\{\alpha_j\}_{j=1}^{\infty}$ $(0<\alpha_1<\alpha_2<\cdots\rightarrow\infty)$ such that $\tx'(\frac{1}{mq}\log\alpha_j)=x'(0)=0$.
Since $v'(\frac{\pi}{2})=Am^{-1}x'(0)$, $v'(\frac{\pi}{2})=0$ if $\alpha=\alpha_j$.
Since $N-3-2q>0$ and $-\infty<s<\frac{1}{mq}\log\alpha$, (\ref{S4E2}) has the Lyapunov function $I(\tx,\ty)=\frac{\ty^2}{2}-\frac{\tx^2}{2}+\frac{\tx^{p+1}}{p+1}$.
We see that $(\tx(s),\ty(s))\in\{I<0\}$.
Thus, $\tx>0$ and $v(\theta)$ is positive.
We have found infinitely many positive solutions of (\ref{S4E1}) with $v'(\frac{\pi}{2})=0$.
\end{proof}

\noindent
{\bf Acknowledgment}\\
The author is grateful to Professor Toru Kan for a stimulating discussion and for bringing his attention to the problem.
He thanks to the referee for the encouraging and helpful comments.
Those improved the presentation of the paper.

% **********************************************************
% **********************************************************
% **********************************************************
% Bibliography
% **********************************************************
% **********************************************************
% **********************************************************

\end{document}